\documentclass[12pt]{amsart}
\usepackage{times,mathrsfs,color,comment}
\usepackage{amsmath}
\usepackage{amsfonts}
\usepackage{amssymb}
\usepackage{latexsym}
\usepackage{amsthm}
\usepackage{a4wide}

\usepackage{tikz}
\usetikzlibrary{calc,patterns}
\usetikzlibrary{shapes}
\usetikzlibrary{snakes}

\newtheorem{theorem}{Theorem}[section]
\newtheorem{lemma}[theorem]{Lemma}

\newcounter{other}            
\newtheorem{otherth}[other]{Theorem}              
\newtheorem{otherl}[other]{Lemma}        

\newcommand{\Cn}{\mathbb{C}^n}

\newcommand{\Sn}{\mathbb{S}_ n}

\newcommand{\Bn}{\mathbb{B}_ n}
\newcommand{\D}{\Bbb D}

\numberwithin{equation}{section}
\begin{document}

\title[Area operators on Bergman spaces]{Boundedness of {area operators} on Bergman spaces}
\author[Xiaofen Lv]{Xiaofen Lv}
\address{Xiaofen Lv \\Department of Mathematics\\ Huzhou University, Huzhou
313000,   China } \email{lvxf@zjhu.edu.cn}

\author[Jordi Pau]{Jordi Pau}
\address{Jordi Pau \\Departament de Matem\`{a}tiques i Inform\`{a}tica \\
Universitat de Barcelona\\
08007 Barcelona\\
Catalonia, Spain} \email{jordi.pau@ub.edu}

\author[Maofa Wang]{Maofa Wang}
\address{Maofa Wang \\School of Mathematics and Statistics, Wuhan University, Wuhan
430072,   China} \email{mfwang.math@whu.edu.cn}



%
\subjclass[2010]{32A35, 32A36, 47B38}

\keywords{Area operator, Bergman space, tent space, area formula}

\thanks{The first author was partially
supported by  {NSFC (11601149,  11771139) and ZJNSF (LY20A010008)}.
The second author is
 supported by the grants MTM2017-83499-P (Ministerio de Educaci\'{o}n y Ciencia)  and  2017SGR358 (Generalitat de Catalunya).  M. Wang was partially supported by  {NSFC (11771340)}.}


\begin{abstract}
 We completely characterize the boundedness of the area operators from the Bergman spaces $A^p_\alpha(\Bn)$ to the Lebesgue spaces $L^q(\Sn)$ for all $0<p,q<\infty$. For the case $n=1$, some partial results were previously obtained by Wu in \cite{Wu11}. Especially, in the case $q<p$ and $q<s$,  we obtain the new characterizations for the  area operators to be bounded.  We solve the cases left open there and extend the results to $n$-complex dimension.
\end{abstract}

\maketitle


\section{Introduction}\allowdisplaybreaks[4]

Let $\Bn$ be the open unit ball in $\Cn$ and $\Sn$ be the unit sphere. For a parameter $\alpha>-1$ we consider the weighted volume measure
$$
dV_\alpha(z)=C(n,\alpha)(1-|z|^2)^\alpha\,dV(z),
$$
where $dV$ is the Lebesgue volume measure
on $\Bn$, and $C(n,\alpha)$ is a constant such that $V_\alpha(\Bn)=1$.  $H(\Bn)$ denotes the set of all holomorphic functions on $\Bn$.
For $0<p<\infty$, a function $f\in H(\Bn)$ belongs to the Bergman space $A^p_\alpha$, if
$$
\|f\|_{A^p_\alpha}=\left(\int_{\Bn}|f(z)|^pdV_\alpha(z)\right)^{\frac{1}{p}}<\infty.
$$
The theory of Bergman spaces and operators acting on them has been an intensive area of research for the las $30$ years, and we refer to the books \cite{ZZ, ZhuBn, ZOT} for more information on this. An important tool for studying operators on Bergman spaces are the Carleson-type measures, that are crucial for describing the boundedness of the embedding $I:A^p_{\alpha}\rightarrow L^p(\Bn,\mu)$. Closely related to that, is the area operator that we are going to define next. First,
for $\gamma>1$ and $\zeta \in \Sn$, define the Kor\'anyi (admissible, non-tangential) approach region $\Gamma_\gamma(\zeta)$ by
\begin{equation}\label{domain-1}
\Gamma_\gamma(\zeta)=\left\lbrace z \in \Bn : |1-\langle z,\zeta\rangle| <\frac{\gamma}{2}(1-|z|^2)\right\rbrace.
\end{equation}
In this paper we agree that $\Gamma(\zeta):=\Gamma_2(\zeta)$.
Let $\mu$ be a positive Borel measure on $\Bn$ and $s>0$. The area operator $A_{\mu,s}$ acting on $H(\Bn)$ is the sublinear operator defined by
$$
A_{\mu,s}(f)(\zeta)=\left(\int_{\Gamma(\zeta)}|f(z)|^s\frac{d\mu(z)}{(1-|z|)^n}\right)^{1/s}.
$$

Area operators are very useful in harmonic analysis. They are related to, for example, the non-tangential maximal functions, Littlewood-Paley operators, multipliers,  Poisson integrals and tent spaces, etc.
On the Hardy spaces $H^p$, for $n=1$ and $s=1$, the area operators have been studied by Cohn \cite{Co97}. The main result in \cite{Co97} states that, for $0<p<\infty$, $A_{\mu,1}: H^p\to L^p(\partial \D)$ is bounded if and only if $\mu$ is a Carleson measure. This result was generalized from $H^p$ to $L^q(\partial \D)$ for $0<p \leq q<\infty$
and $1\leq q<p\leq \infty$ in \cite{GLW}. In the unit disc, Wu \cite{Wu06} discussed the boundedness of $A_{\mu,1}$ from Bergman spaces $A^p_\alpha$ to $L^q(\partial \D)$ for $0<p \leq q<\infty$
and $1\leq q<p\leq \infty$. Wu also mentioned the characterizations on $\mu$ for $A_{\mu,s}$ to be bounded from $A^p_\alpha$ to $L^q(\partial \D)$ for $0<s\leq q<p\leq\infty$ in \cite{Wu11}, but the case $q<s$ was left open. It is our goal to solve this last case, and to extend the results to the setting of the unit ball.

Let $\beta(z,w)$ denote the Bergman metric on $\Bn$, and $D(a,r)=\{z\in \Bn : \beta(a,z)<r\}$ be the Bergman metric ball of radius $r>0$ centered at a point $a\in \Bn$. Given a positive Borel measure $\mu$ on $\Bn$ and $\alpha>-1$, for $z\in \Bn$ and $r>0$, we define
$$\rho_{\alpha,\mu}(z, r)=\frac{\mu(D(z,r))}{V_\alpha(D(z,r))}.$$
The following is our main result which completely characterizes the boundedness of the area operators from the Bergman spaces $A^p_\alpha$ to the Lebesgue spaces $L^q(\Sn)$ by $\rho_{\alpha,\mu}(z, r)$ for all $0<p,q<\infty$. The notion of $\beta$-Carleson measures are defined in section \ref{pre}.

\begin{theorem}\label{th-1}
Let $\alpha>-1$, $0<r, s, p, q<\infty$, and let $\mu$ be a positive Borel measure on $\Bn$. Then the following statements hold.
\begin{enumerate}
\item[(a)] If $p\leq \min\{s,q\}$, then $A_{\mu,s}: A_\alpha^p\rightarrow L^q (\Sn)$ is bounded if and only if
$$
\frac{\rho_{\alpha,\mu}(z,r)}
{(1-|z|)^{\frac{(n+1+\alpha)(s-p)}{p}
+\frac{n(q-s)}{q}}}
$$
is bounded in $\Bn$.
\vspace{1mm}
	
\item[(b)]If $s<p\leq q$,  then  $A_{\mu,s}: A_\alpha^p\rightarrow L^q (\Sn)$ is bounded if and only if
$$
\rho_{\alpha,\mu}(z,r)^{\frac{p}{p-s}}dV_\alpha(z)
$$
 is a $\frac{p(q-s)}{q(p-s)}$-Carleson measure.
 \vspace{1mm}	
 	
\item[(c)]  If $p> \max\{s,q\}$,   then $A_{\mu,s}: A_\alpha^p\rightarrow L^q (\Sn)$ is bounded if and only if
$$
 \int_{\Gamma(\zeta)}\rho_{\alpha,\mu}(z,r)^{\frac{p}{p-s}}\frac{dV_\alpha(z)}{(1-|z|)^n}
$$
belongs to $L^{\frac{q(p-s)}{s(p-q)}}(\Sn)$.
\vspace{1mm}

 \item[(d)] If $q<p\leq s$,   then $A_{\mu,s}: A_\alpha^p\rightarrow L^q (\Sn)$ is bounded if and only if
 $$
 \sup\limits_{z\in \Gamma(\zeta)}\rho_{\alpha,\mu}(z,r)(1-|z|)^{\frac{(\alpha+1)(s-p)}{p}}
 $$
belongs to $L^{\frac{pq}{s(p-q)}}(\Sn).$
\end{enumerate}
\end{theorem}
\medskip
The most important part of this result is the case $q<p$, that was left open even for the unit disk in. The proof of this case given here, uses several tools such as Khintchine-Kahane's inequalities, and factorization of tent sequence spaces. Moreover, the proof we give of the case $p\le q$ is, in our opinion, more simpler and natural than the known proof for the case of the unit disk.\\

The exposition of the paper is organized as follows.
In Section \ref{pre} we recall some basic facts which will be used in the sequel. Section \ref{Main} is devoted to the proof of Theorem \ref{th-1}.\medskip

\textit{Conventions.}
Throughout the paper, constants are
often given without computing their exact values, and the value of a constant $C$ may change
from one occurrence to the next. We also use the notation
$a\lesssim b$  or $b\gtrsim a$ for nonnegative quantities $a$ and
$a$ to indicate that $a\le C b$ for some inessential constant $C>0$.
Similarly, we use the notation $a\asymp b$ if both $a\lesssim b$
and $b \lesssim a$ hold. Given $p \in [1,\infty]$, we will denote by $p'=p/(p-1)$ its H\"older conjugate. In this context we agree that $1'=\infty$ and $\infty'=1$.\bigskip

\section{Preliminaries}\label{pre}

In this section, we collect the necessary preliminaries for the course of the proof of our main theorem.

\subsection{Carleson measures and embedding theorems}

Let us recall the concept of a Carleson measure. For $\zeta \in \Sn$ and $\delta>0$, consider the non-isotropic metric ball:
\[
B_{\delta}(\zeta)=\left\lbrace z\in \Bn: |1-\langle z, \zeta \rangle |<\delta \right\rbrace.
\]
A positive Borel measure $\mu$ on $\Bn$ is said to be a Carleson measure if
\[
\mu(B_{\delta}(\zeta))\lesssim \delta^ n
\]
for all $\zeta \in \Sn$ and $\delta>0$. It is clear that every Carleson measure is finite. In \cite{H} H\"{o}rmander  generalized to the setting of several complex variables the famous  Carleson measure embedding theorem \cite{Car0, Car1} asserting that, for $0<p<\infty$, the embedding $I_ d:H^p\rightarrow L^p(\mu):=L^p(\Bn,d\mu)$ is bounded if and only if $\mu$ is a Carleson measure.
More generally, for $s>0$, a finite positive Borel measure $\mu$ on $\Bn$ is  an $s$-Carleson measure if  $\mu(B_{\delta}(\zeta))\lesssim \delta^{ns}$ for all $\zeta\in \Sn$ and $\delta>0$. We set $$\|\mu\|_{CM_s}:=\sup_{\zeta \in \Sn, \delta>0} \mu(B_{\delta}(\zeta))\delta^{-ns}.$$
For simplicity, we also write $\|\mu\|_{CM}$ for  $\|\mu\|_{CM_1}.$
It is well-known (see \cite[Theorem 45]{ZZ}) that $\mu$ is an $s$-Carleson measure if and only if, for each (some) $t>0$,
\begin{equation}\label{sCM}
\sup_{a\in \Bn}\int_{\Bn} \!\!\frac{(1-|a|^2)^t}{|1-\langle a,z \rangle |^{ns+t}} \,d\mu(z)<\infty.
\end{equation}
Moreover, with constants depending on $t$, the supremum of the above integral is comparable to $\|\mu\|_{CM_s}$.
In \cite{Du}, Duren gave an extension of Carleson's theorem by showing that, for
 $0<p\leq q<\infty$, one has that $I_ d:H^p\rightarrow L^q(\mu)$ is bounded if and only if $\mu$ is a $q/p$-Carleson measure. Moreover, one has the estimate
 \begin{equation}\label{sCM-1}
 \|I_ d\|_{H^p\rightarrow L^q(\mu)}\asymp \|\mu\|^{1/q}_{CM_{q/p}}.
 \end{equation}
 A  proof of this result, in the setting of the unit ball, can be found in \cite{P1} for example.

Given a positive Borel measure $\mu$ on $\Bn$, then for $\zeta \in \Sn$, we define
$$\widetilde{\mu}(\zeta)=\int_{\Gamma(\zeta)}\frac{d\mu(z)}{(1-|z|^2)^n}.$$

\noindent The following result is known as Luecking's theorem, and is originally from \cite{Lue1}. The present variant can be found in \cite{P1}, for instance.

\begin{otherth}\label{LT}
Let $0<s<p<\infty$ and let $\mu$ be a positive Borel measure on $\Bn$. Then the identity $I_ d:H^p\rightarrow L^s(\mu)$ is bounded, if and only if, the function $\widetilde{\mu}$
belongs to $L^{p/(p-s)}(\Sn)$. Moreover,  $\|I_d\|_{H^p\rightarrow L^s(\mu)}\asymp \|\widetilde{\mu}\|_{L^{p/(p-s)}(\Sn)}^{1/s}.$
\end{otherth}

\subsection{Separated sequences and lattices}

A sequence $Z=\{a_k\}\subset \Bn$ is said to be separated (in the Bergman metric) if there exists $\delta>0$ such that $\beta(a_ i,a_ j)\geq \delta$ for all $i$ and $j$ with $i\neq j$. This implies that there is $r>0$ such that the Bergman metric balls $D(a_k,r)=\{z\in \Bn :\beta(z,a_k)<r\}$ are pairwise disjoint.

We are going to use a well-known
result on decomposition of the unit ball $\Bn$.
By Theorem 2.23 in \cite{ZhuBn},
there exists a positive integer $N$ such that for any $0<r\le1$ we can
find a sequence $\{a_k\}$ in $\Bn$ with the  properties:
\begin{itemize}
\item[(i)  ] $\Bn=\bigcup_{k}D(a_k,r)$;
\item[(ii) ] The sets $D(a_k,r/4)$ are mutually disjoint;
\item[(iii) ] Each point $z\in\Bn$ belongs to at most $N$ of the sets $D(a_k,4r)$.
\end{itemize}

\noindent Any sequence $\{a_k\}$ satisfying the above conditions is called
an $r$-\emph{lattice}
(in the Bergman metric). Obviously any $r$-lattice is a separated sequence.  We often write $D_k=D(a_k, r)$ and $\widetilde{D}_k=D(a_k, 2r)$ for simplicity.

We need the following important result, essentially due to Coifman and Rochberg \cite{CR}, that can be found in Theorem 2.30 of \cite{ZhuBn}. We only need one part of the cited theorem, and it is easy to see that this part holds for all separated sequences.

\begin{otherth}\label{atomic}
Let $0<p<\infty$, $\alpha>-1$, and
$s>n\max\{1,1/p\}-n/p.$
For any separated sequence $\{a_k\}$ and $\lambda=\{\lambda_k\} \in \ell^p$, the function
$$f(z)=\sum_{k}\lambda_k \frac{(1-|a_k|^2)^s}{(1-\langle z,a_k\rangle)^{s+\frac{n+1+\alpha}{p}}}$$
belongs to $A^p_\alpha$, where the series converges in the quasinorm topology of $A^p_\alpha$. Moreover,
$\|f\|_{A^p_\alpha}\lesssim \|\lambda\|_{\ell^p}.$
\end{otherth}

For $\gamma>1$ and $\zeta \in \Sn$, recall that $\Gamma_\gamma(\zeta)$ is the Kor\'anyi (admissible, non-tangential) approach region defined by \eqref{domain-1}. It is known that for every $r\le 1$ and $\gamma>1$, there exists $\gamma'>1$ so that
\begin{equation}\label{setin}
\bigcup_{z \in \Gamma_\gamma(\zeta)}D(z,r)\subset \Gamma_{\gamma'}(\zeta).
\end{equation}
We will write $\widetilde{\Gamma}(\zeta)$ (and sometimes $\widetilde{\widetilde{\Gamma}}(\zeta)$) to indicate this change of aperture.

Given $z \in \Bn$, we define the set
$$
I(z)=\{\zeta \in \Sn : z \in \Gamma(\zeta)\}\subset \Sn.
$$
 Since $\sigma(I(z))\asymp (1-|z|^2)^n$, an application of Fubini's theorem yields the  important estimate
\begin{equation}\label{EqG}
\int_{\Bn} \varphi(z)d\nu(z)\asymp \int_{\Sn} \left (\int_{\Gamma(\zeta)} \varphi(z) \frac{d\nu(z)}{(1-|z|^2)^{n}} \right )d\sigma(\zeta),
\end{equation}
where $\varphi$ is any positive measurable function and $\nu$ is a finite positive measure.

\subsection{Kahane-Khinchine inequalities}

Consider a sequence of Rademacher functions $r_ k(t)$ (see \cite[Appendix A]{duren1} for example). For almost every $t\in (0,1)$ the sequence $\{\gamma_ k(t)\}$ consists of signs $\pm 1$. We state first the classical Khinchine's inequality (see \cite[Appendix A]{duren1} for example).\medskip

\noindent \textbf{Khinchine's inequality:} Let $0<p<\infty$. Then for any sequence $\{c_k\}$ of complex numbers, we have
\begin{equation}\label{khin}
\left(\sum_k |c_k|^2\right)^{p/2}\asymp \int_0^1 \left|\sum_k c_kr_k(t)\right|^p dt.
\end{equation}

The next result, known as Kahane's inequality,  will be usually applied in connection to Khinchine's inequality. For reference, see Lemma 5 of Luecking \cite{Lue2} or the paper of Kalton \cite{Kal}.\medskip

\noindent \textbf{Kahane's inequality:} Let $X$ be a quasi-Banach space,  and $0<p,q<\infty$. For any sequence $\{x_ k\}\subset X$, one has
\begin{equation}\label{kah}
\left ( \int_{0}^1 \Big \|\sum_ k r_ k(t)\, x_ k \Big \|_{X}^q dt\right )^{1/q} \asymp \left ( \int_{0}^1 \Big \|\sum_ k r_ k(t)\, x_ k \Big \|_{X}^p dt\right )^{1/p}.
\end{equation}
Moreover, the implicit constants can be chosen to depend only on $p$ and $q$, and not on the quasi-Banach space $X$.

\subsection{Tent spaces}

Tent spaces were introduced by Coifman, Meyer and Stein \cite{CMS} in order to study several problems in harmonic analysis, and provide us us with a general framework for questions regarding important spaces such Hardy spaces, Bergman spaces, or  BMOA among others. Luecking \cite{Lue1} used these tent spaces to study embedding theorems for Hardy spaces on $\mathbb{R}^n$,   results that have been obtained in the unit ball $\Bn$ by Arsenovic and Jevtic \cite{Ars, Jev}.

 Let $0<p,q<\infty$, and $Z=\{a_k\}$ be a separated sequence. The tent sequence space $T^p_ q(Z)$ consists of those sequences $\lambda=\{\lambda_k\}$ with
$$\|\lambda\|_{T^{p}_q(Z)}^p : =\int_{\Sn} \left(\sum_{a_k \in \Gamma(\zeta)}|\lambda_k|^q  \right)^{\frac{p}{q}}d\sigma(\zeta)<\infty.$$
Also, $\lambda=\{\lambda_k\} \in T^{p}_\infty (Z)$, if
$$\|\lambda\|_{T^{p}_\infty (Z)}^p : =\int_{\Sn}\left(\sup_{a_k \in \Gamma(\zeta)}|\lambda_k|\right)^p d\sigma(\zeta)<\infty.$$

\noindent Finally, $\lambda=\{\lambda_k\} \in T^{\infty}_q(Z)$, if
$$\|\lambda\|_{T_{q}^\infty(Z)}= \sup_{\zeta\in\Sn}\left(
\sup_{u\in\Gamma(\zeta)}{1\over (1-|u|^2)^n}\sum_{a_k \in Q(u)}|\lambda_k|^q (1-|a_k|^2)^{n}\right)^{1/q}<\infty.$$
We have that $\lambda \in T^{\infty}_q(Z)$ if and only if the measure $d\mu_{\lambda}=\sum_ k |\lambda_ k|^q (1-|a_ k|^2)^n \delta_{a_ k}$ is a Carleson measure. Moreover,  $\|\lambda\|_{T_{q}^\infty(Z)}\asymp \|\mu_{\lambda}\|_{CM_ 1}^{1/q}$. \\

It is well known that different apertures of the Korany approach region, define the same tent sequence  spaces with equivalent quasinorms.\\

\noindent We  need the following duality result for  tent sequence spaces. For the proof, see \cite{Ars, Jev, Lue1}.

\begin{otherth}\label{TTD1}
Let $1<p<\infty$ and $Z=\{a_k\}$ be a separated sequence. If $1<q<\infty$, then the dual of $T^p_q(Z)$ is isomorphic to $T^{p'}_{q'}(Z)$ under the pairing:
$$\langle \lambda ,\mu \rangle_{T^2_2(Z)} =\sum_ k \lambda _ k \,\overline{\mu_ k} (1-|a_ k|^2)^n,\quad \lambda \in T^p_q(Z),\quad \mu \in T^{p'}_{q'}(Z).$$
If $0<q\leq 1$, then the dual of $T^p_q(Z)$ is isomorphic to $T^{p'}_\infty(Z)$ under the same pairing.
\end{otherth}

We will use a result concerning factorization of tent sequence spaces, which can be found in \cite{MPPW}. A similar result for tent spaces of functions over the upper half-space could be found in \cite{CV}.

\begin{otherth}\label{TTD2}
Let $0<p,q<\infty$ and $Z=\{a_k\}$ be an $r$-lattice. If $p<p_1, p_2<\infty$,  $q<q_1,q_2<\infty$ and satisfy
$$
\frac{1}{p}=\frac{1}{p_1}+\frac{1}{p_2},\qquad \frac{1}{q}=\frac{1}{q_1}+\frac{1}{q_2}.
$$
Then
$$
T^p_q(Z)=T^{p_1}_{q_1}(Z)\cdot T^{p_2}_{q_2}(Z).
$$
That is, if $\alpha\in T^{p_ 1}_{q_ 1}(Z)$ and $\beta \in T^{p_2}_{q_2}(Z)$, then $\alpha\cdot \beta \in T^p_q(Z)$ with $\|\alpha \cdot \beta \|_{T^p_q(Z)}\lesssim \|\alpha\|_{T^{p_ 1}_{q_ 1}(Z)}\cdot \|\beta \|_{T^{p_2}_{q_2}(Z)}$; and conversely, if $\lambda \in T^p_ q(Z)$, then there are sequences $\alpha \in T^{p_ 1}_{q_ 1}(Z)$ and $\beta \in T^{p_2}_{q_2}(Z)$ such that $\lambda=\alpha \cdot \beta$, and $\|\alpha\|_{T^{p_ 1}_{q_ 1}(Z)}\cdot \|\beta \|_{T^{p_2}_{q_2}(Z)}\lesssim \|\lambda \|_{T^p_q(Z)}$.
\end{otherth}

We usually are going to obtain first the  discrete versions of the conditions we really need.  The next results are useful in order to get the continuous characterizations from the discrete ones.

\begin{lemma}\label{discrete-1}
Let $0<p,q<\infty$, $\beta>-1$  and let $\mu$ be a positive measure. Then there exists $r_0\in (0,1)$ such that for $0<r<r_0$ and any $r$-lattice $\{a_k\}$,
\begin{eqnarray*}
&&\int_{\Sn}\left(\int_{\Gamma(\zeta)}\rho_{\alpha,\mu}(z,r)^p (1-|z|)^\beta dV(z)\right)^q d\sigma(\zeta)
\\
&&\verb#   #\lesssim \int_{\Sn}\left(\sum_{a_k\in \Gamma(\zeta)} \rho_{\alpha,\mu}(a_k,2r)^p (1-|a_k|)^{n+1+\beta}\right)^qd\sigma(\zeta).
\end{eqnarray*}
\end{lemma}

\begin{proof}
Recall $D_k=D(a_k, r)$ and  $\widetilde{D}_k=D(a_k, 2r)$.  Then
\begin{eqnarray*}
 && \int_{\Gamma(\zeta)}\rho_{\alpha,\mu}(z,r)^p (1-|z|)^\beta dV(z)
 \\&&\verb#   #\leq \sum_{k: D_k\bigcap\Gamma(\zeta)\neq\emptyset}\int_{D_k}\rho_{\alpha,\mu}(z,r)^p (1-|z|)^\beta dV(z)
 \\
   && \verb#   #\lesssim  \sum_{k: D_k\bigcap\Gamma(\zeta)\neq\emptyset}(1-|a_k|)^{n+1+\beta}\sup_{z\in D_k}\rho_{\alpha,\mu}(z,r)^p
   \\
   && \verb#   #\lesssim \sum_{a_k\in \widetilde{\Gamma}(\zeta)}(1-|a_k|)^{n+1+\beta}\rho_{\alpha,\mu}(a_k,2r) ^p.
\end{eqnarray*}
This yields the desired result, because  different apertures define the same tent spaces with equivalent quasinorms.
\end{proof}

By a similar argument, we obtain the following result.

\begin{lemma}\label{discrete-2}
Let $0<p,q<\infty$, $\beta>-1$ and let $\mu$ be a positive measure. Then there exists $r_0\in (0,1)$ such that for $0<r<r_0$ and any $r$-lattice $\{a_k\}$,
\begin{eqnarray*}
\int_{\Sn}\sup_{z\in \Gamma(\zeta)}\rho_{\alpha,\mu}(z,r)^p (1-|z|)^\beta d\sigma(\zeta)
\lesssim
\int_{\Sn}\sup_{a_k\in \Gamma(\zeta)}\rho_{\alpha,\mu}(a_k,2r)^p (1-|a_k|)^\beta d\sigma(\zeta).
\end{eqnarray*}
\end{lemma}

\subsection{Estimates involving Bergman kernels}

We also need the following well-known Forelli-Rudin integral type estimates, that have been very useful in this area of analysis (see \cite[Theorem 1.12]{ZhuBn} for example).
\begin{otherl}\label{IctBn}
Let $t>-1$ and $s>0$. Then
$$\int_{\Sn} \frac{d\sigma(\zeta)}{|1-\langle z,\zeta\rangle |^{n+s}} \lesssim (1-|z|^2)^{-s}$$
and
$$ \int_{\Bn} \frac{(1-|u|^2)^t\,dV(u)}{|1-\langle z,u\rangle |^{n+1+t+s}}\lesssim (1-|z|^2)^{-s}$$
for all $z\in \Bn$.
\end{otherl}







\section{The proof of main result}\label{Main}

We will separate   Theorem \ref{th-1} into  four results according to what case is considered.

\begin{theorem}\label{th-01}
Let $\alpha>-1$, $0<p\leq \min\{s,q\}<\infty$, and let $\mu$ be a positive Borel measure on $\Bn$. Then $A_{\mu,s}: A_\alpha^p\rightarrow L^q (\Sn)$ is bounded if and only if
$$
F_\mu(z):=\frac{\rho_{\alpha,\mu}(z,t)}{(1-|z|)^{\frac{(n+1+\alpha)(s-p)}{p}+\frac{n(q-s)}{q}}}
$$
is bounded in $\Bn$. Moreover, $\|A_{\mu,s}\|^s_{A_\alpha^p\rightarrow L^q (\Sn)}\asymp \sup\limits_{z\in \Bn}F_\mu(z)$.
\end{theorem}

\begin{proof}
We first consider the necessity.  Fixing any $a\in \Bn$, we take
\[
f_a(z)=\frac{(1-|a|^2)^{m}}{(1-\langle z,a\rangle)^{m+\frac{n+1+\alpha}{p}}}
\]
where  $m>0$. It is easy to see that $\|f_a\|_{A_\alpha^p}\asymp 1$ and
\[
\big |f_a(z)\big |\asymp \frac{1}{(1-|a|^2)^{\frac{n+1+\alpha}{p}}}, \qquad z\in D(a, r)
\]
for $r>0$.
 Due to \eqref{setin} and the fact that $A_{\mu,s}$ is bounded from $A_\alpha^p$ to $L^q (\Sn)$, we get
\begin{eqnarray*}
 && (1-|a|^2)^n\left(\frac{\mu(D(a,r))}{ (1-|a|^2)^{n+\frac{(n+1+\alpha)s}{p}}}\right)^{q/s} \asymp \int_{I(a)}\left(\int_{D(a,r)}|f_a(z)|^s\frac{d\mu(z)}{(1-|z|^2)^n}\right)^{q/s}d\sigma(\zeta)
 \\
   && \verb#                         #\le   C\int_{\Sn}\left(\int_{\widetilde{\Gamma}(\zeta)}|f_a(z)|^s\frac{d\mu(z)}{(1-|z|)^n}\right)^{q/s}d\sigma(\zeta)
  \\
   && \verb#                         #\asymp\int_{\Sn}\left(\int_{\Gamma(\zeta)}|f_a(z)|^s\frac{d\mu(z)}{(1-|z|)^n}\right)^{q/s}d\sigma(\zeta)
    \\
   && \verb#                         #\le C\|A_{\mu,s}\|^q_{A_\alpha^p\rightarrow L^q (\Sn)}.
\end{eqnarray*}
This gives that
\begin{equation}\label{est-1}
\sup_{a\in \Bn}F_\mu(a)\leq C\|A_{\mu,s}\|^s_{A_\alpha^p\rightarrow L^q (\Sn)}.
\end{equation}

We next look for the sufficiency. For $f\in A^p_\alpha$, \eqref{setin} implies
  \begin{equation}\label{est-3}
  |A_{\mu,s}f(\zeta)|^s\leq C\int_{\widetilde{\Gamma}(\zeta)}|f(z)|^s\rho_{\alpha}(\mu)(z,r)\frac{d V_\alpha(z)}{(1-|z|)^n}.
   \end{equation}
Let $\{a_k\}$
be a $r$-lattice in $\Bn$. Since $\frac{p}{s}\leq 1$, we have
\begin{eqnarray*}
 && |A_{\mu,s}f(\zeta)|^p\le     C\left(\sup_{z\in \Bn}F_\mu(z)\sum_{k:D_k\bigcap\widetilde{\Gamma}(\zeta)\neq\emptyset}(1-|a_k|)^{\frac{(n+1+\alpha)s}{p}-\frac{ns}{q}}\sup_{z\in D_k}|f(z)|^s\right)^{\frac{p}{s}}
 \\
 &&\verb#       #
   \leq C\left(\sup_{z\in \Bn}F_\mu(z)\right)^{\frac{p}{s}}\sum_{k:D_k\bigcap\widetilde{\Gamma}(\zeta)\neq\emptyset}(1-|a_k|)^{n+1+\alpha-\frac{np}{q}}\sup_{z\in D_k}|f(z)|^p
\\
 &&\verb#       #
   \leq C\left(\sup_{z\in \Bn}F_\mu(z)\right)^{\frac{p}{s}}\sum_{k:D_k\bigcap\widetilde{\Gamma}(\zeta)\neq\emptyset}\int_{ \widetilde{D}_k}(1-|w|)^{-\frac{np}{q}}|f(w)|^pdV_\alpha(w)
\\
 &&\verb#       #
   \leq CN\left(\sup_{z\in \Bn}F_\mu(z)\right)^{\frac{p}{s}}\int_{\widetilde{\widetilde{\Gamma}}(\zeta)}(1-|w|)^{-\frac{np}{q}}|f(w)|^pdV_\alpha(w).
   \end{eqnarray*}
If $q=p$, then \eqref{EqG} shows
\begin{eqnarray*}
 && \|A_{\mu,s} f\|^q_{L^q(\Sn)} \leq    C\left(\sup_{z\in \Bn}F_\mu(z)\right)^{\frac{p}{s}}\int_{\Sn}\left(\int_{\Gamma(\zeta)}|f(w)|^p\frac{dV_\alpha(w)}{(1-|w|)^{n}}\right)d\sigma(\zeta)
 \\&&\verb#      # \leq C\left(\sup_{z\in \Bn}F_\mu(z)\right)^{\frac{p}{s}}\|f\|_{A_\alpha^p}^p.
   \end{eqnarray*}
If $q>p$, then $\frac{q}{p}>1$. By H\"{o}lder's inequality, we obtain
\begin{eqnarray*}
 && |A_{\mu,s} f(\zeta)|^q\leq    C\left(\sup_{z\in \Bn}F_\mu(z)\right)^{\frac{q}{s}}\left(\int_{\widetilde{\widetilde{\Gamma}}(\zeta)}(1-|w|)^{-\frac{np}{q}}|f(w)|^pdV_\alpha(w)\right)^{\frac{q}{p}}
 \\
 &&\verb#       #
   \leq C
   \left(\sup_{z\in \Bn}F_\mu(z)\right)^{\frac{q}{s}}\int_{\widetilde{\widetilde{\Gamma}}(\zeta)}|f(w)|^p\frac{dV_\alpha(w)}{(1-|w|)^{n}}\left(\int_{\widetilde{\widetilde{\Gamma}}(\zeta)}|f(w)|^pdV_\alpha(w)\right)^{\frac{q-p}{p}}.
  \\
 &&\verb#       #
  \leq  C\left(\sup_{z\in \Bn}F_\mu(z)\right)^{\frac{q}{s}}\|f\|_{A_\alpha^p}^{q-p}\int_{\widetilde{\widetilde{\Gamma}}(\zeta)}|f(w)|^p\frac{dV_\alpha(w)}{(1-|w|)^{n}}.
   \end{eqnarray*}
This, together with  \eqref{EqG}, implies
\[
 \begin{split}
 \|A_{\mu,s} f\|^q_{L^q(\Sn)}&\leq C\left(\sup_{z\in \Bn}F_\mu(z)\right)^{\frac{q}{s}}\|f\|_{A_\alpha^p}^{q-p} \int_{\Sn}\left(\int_{\Gamma(\zeta)}|f(w)|^p\frac{dV_\alpha(w)}{(1-|w|)^{n}}\right)d\sigma(\zeta)
 \\
  &\leq C\left(\sup_{z\in \Bn}F_\mu(z)\right)^{\frac{q}{s}}\|f\|_{A_\alpha^p}^{q}.
\end{split}
\]
  Therefore, $A_{\mu,s}$ is bounded from $A_\alpha^p$ to $L^q (\Sn)$, and
\[
\|A_{\mu,s}\|^s_{A_\alpha^p\rightarrow L^q (\Sn)}\leq C\sup_{a\in \Bn}F_\mu(z).
\]
 This, combined with  \eqref{est-1},  tells us that
\[
\|A_{\mu,s}\|^s_{A_\alpha^p\rightarrow L^q (\Sn)}\asymp\sup_{a\in \Bn}F_\mu(z).
\]
\end{proof}
Before proving part (b) in Theorem \ref{th-1}, we need a preliminary lemma.

\begin{lemma}\label{PLem}
Let $\alpha>-1$, $0< s<p\leq q<\infty$, and $\tau>0$. Let $\mu$ be a positive Borel measure on $\Bn$. The following conditions are equivalent:
\begin{itemize}
\item[(a)] The measure $d\nu(z):=\rho_{\alpha,\mu}(z,r)^{\frac{p}{p-s}}dV_\alpha(z)
$
 is a $\frac{p\tau}{(p-s)}$-Carleson measure.

\item[(b)] For any $f\in A^p_{\alpha}$ and $\beta>0$, we have
\[
\sup_{a\in \Bn}\int_{\Bn}\!\!\frac{(1-|a|^2)^{\beta}}{|1-\langle a,z \rangle |^{n\tau+\beta}} \,|f(z)|^s\,d\mu(z)\lesssim \|f\|_{A^p_{\alpha}}^s.
\]
\end{itemize}
Moreover, we have
\[
\big \| \nu \big \|_{CM_{\frac{p\tau}{(p-s)}}} \asymp \sup _{\|f\|_{A^p_{\alpha} \asymp 1}} \Big \| |f|^s d\mu \Big \|^{p/(p-s)}_{CM_ {\tau}}.
\]
\end{lemma}

\begin{proof}
For $a\in \Bn$, consider the measure $\mu_ a$ defined as
\[
d\mu_ a(z)=\frac{(1-|a|^2)^{\beta}}{|1-\langle a,z \rangle |^{n\tau+\beta}} \,d\mu(z)
\]
With this notation, condition (b) corresponds to the inequality
\[
\sup_{a\in \Bn}\int_{\Bn}\!|f(z)|^s\,d\mu_ a (z)\lesssim \|f\|_{A^p_{\alpha}}^s,
\]
and, since $s<p$,  by Theorem 54 in \cite{ZZ}, this is equivalent to the function $\rho_{\alpha,\mu_ a}(z,r)$ being in  $L^{\frac{p}{p-s}}(\Bn,dV_{\alpha})$, with
\[
 \Big \| \rho_{\alpha,\mu_ a}(\cdot ,r) \Big \|_{L^{\frac{p}{p-s}}(\Bn,dV_{\alpha})} \asymp \sup _{\|f\|_{A^p_{\alpha} \asymp 1}}\int_{\Bn}\!|f(z)|^s\,d\mu_ a (z),
 \]
 that gives
\[
\sup_{a\in \Bn}\Big \| \rho_{\alpha,\mu_ a}(\cdot ,r) \Big \|_{L^{\frac{p}{p-s}}(\Bn,dV_{\alpha})} \asymp \sup _{\|f\|_{A^p_{\alpha} \asymp 1}} \Big \| |f|^s d\mu \Big \|_{CM_ {\tau}}.
\]
On the other hand,
\[
\mu_ a(D(z,r))\asymp \frac{(1-|a|^2)^{\beta}}{|1-\langle a,z \rangle |^{n\tau+\beta}} \,\mu(D(z,r)),
\]
so that, by  \eqref{sCM},
\[
\sup_{a\in \Bn} \Big \| \rho_{\alpha,\mu_ a}(\cdot ,r) \Big \|^{p/(p-s)}_{L^{\frac{p}{p-s}}(\Bn,dV_{\alpha})} \asymp \sup_{a\in \Bn}\int_{\Bn} \frac{(1-|a|^2)^{\frac{\beta p}{p-s}}}{|1-\langle a,z |^{n\,\frac{p\tau}{p-s} + \frac{\beta p}{p-s}}}\,d\nu(z) \asymp \big \| \nu \big \|_{CM_{\frac{p\tau}{(p-s)}}}.
\]
This finishes the proof.
\end{proof}

The next result is part (b) in Theorem \ref{th-1} with the corresponding estimates for the norm of $A_{\mu,s}$.

\begin{theorem}\label{th-02}
Let $\alpha>-1$, $0< s<p\leq q<\infty$, and let $\mu$ be a positive Borel measure on $\Bn$. Then  $A_{\mu,s}: A_\alpha^p\rightarrow L^q (\Sn)$ is bounded if and only if
\[
d\nu(z):=\rho_{\alpha,\mu}(z,r)^{\frac{p}{p-s}}dV_\alpha(z)
\]
 is a $\frac{p(q-s)}{q(p-s)}$-Carleson measure. Moreover, we have
\[
\big \|A_{\mu,s} \big \|_{A^p_{\alpha} \rightarrow L^q(\Sn)}\asymp \big \| \nu \big \| ^{\frac{p-s}{ps}}_{CM_\frac{p(q-s)}{q(p-s)}}.
\]
\end{theorem}

\begin{proof}  Suppose $A_{\mu,s}:A_\alpha^p\rightarrow L^q (\Sn)$ is bounded. Then $A_{\mu,s} (f)$ belongs to $L^q (\Sn)$ for any $f\in A_\alpha^p$, and therefore  $\left(A_{\mu,s} (f)\right)^s\in  L^{q/s} (\Sn)$. Consider the measure $\mu_{f,s}$ defined by $d\mu_{f,s}(z)=|f(z)|^s\,d\mu(z)$. Then $\big \|(A_{\mu,s} (f)^s\big \| _{L^{q/s}(\Sn)}=\big \|\widetilde{\mu}_{f,s}\big \| _{L^{q/s}(\Sn)}$, and by Theorem \ref{LT} we have
\begin{equation}\label{EqT1b}
\Big \| A_{\mu,s} (f) \Big \|^s_{L^{q}(\Sn)}=\Big \|\left( A_{\mu,s} (f)\right)^s \Big \|_{L^{q/s}(\Sn)} \asymp \sup_{\|g\|_{H^{\frac{q}{q-s}}}\asymp 1} \int_{\Bn} |g(z)|\,|f(z)|^s d\mu(z).
\end{equation}
Fix a point $a\in \Bn$ and $\beta>0$. Consider the holomorphic function $g_ a$ defined as
\[
g_ a(z)=\frac{(1-|a|^2)^{\beta}}{(1-\langle z,a\rangle )^{n\frac{(q-s)}{q}+\beta}},\qquad z\in \Bn.
\]
Then, by Lemma \ref{IctBn}, $g_ a \in H^{\frac{q}{q-s}}$ with $\|g_ a\|_{H^{\frac{q}{q-s}}}\asymp 1$. Hence, we have
\[
\int_{\Bn}\!\!\frac{(1-|a|^2)^{\beta}}{|1-\langle a,z \rangle |^{n\frac{(q-s)}{q}+\beta}} \,|f(z)|^s\,d\mu(z)=\int_{\Bn} |g_ a(z)|\,|f(z)|^s d\mu(z)\lesssim \,\Big \| A_{\mu,s} (f) \Big \|^s_{L^{q}(\Sn)}.
\]
Because of \eqref{sCM}, the measure $d\nu_ f(z)=|f(z)|^s\,d\mu(z)$ is a $\frac{q-s}{q}$-Carleson measure for any $f\in A^p_{\alpha}$, and moreover $\|\nu_ f\|_{CM_{\frac{q-s}{q}}}\lesssim \| A_{\mu,s} (f) \Big \|^s_{L^{q}(\Sn)}$. By Lemma \ref{PLem} it follows that $\nu$ is a $\frac{p(q-s)}{q(p-s)}$-Carleson measure with
\[
\big \| \nu \big \| ^{\frac{p-s}{ps}}_{CM_\frac{p(q-s)}{q(p-s)}}\asymp \sup _{\|f\|_{A^p_{\alpha} \asymp 1}} \Big \| \nu_ f \Big \|^{1/s}_{CM_ {\frac{q-s}{q}}} \lesssim \sup _{\|f\|_{A^p_{\alpha} \asymp 1}}\| A_{\mu,s} (f) \Big \|_{L^{q}(\Sn)} \asymp \big \|A_{\mu,s} \big \|_{A^p_{\alpha} \rightarrow L^q(\Sn)}.
\]

Conversely, if $\nu$ is an $\frac{p(q-s)}{q(p-s)}$-Carleson measure, by Lemma \ref{PLem}, the measure $\nu_ f$ is a $\frac{q-s}{q}$-Carleson measure for any $f\in A^p_{\alpha}$ with
\[
\sup _{\|f\|_{A^p_{\alpha} \asymp 1}} \Big \| \nu_ f \Big \|^{1/s}_{CM_ {\frac{q-s}{q}}} \asymp \big \| \nu \big \| ^{\frac{p-s}{ps}}_{CM_\frac{p(q-s)}{q(p-s)}}.
\]
By the Carleson-H\"{o}rmander-Duren's theorem, we have
\[
\int_{\Bn} |g(z)|\,|f(z)|^s d\mu(z) \lesssim \|g\|_{H^{\frac{q}{q-s}}} \Big \| \nu_ f \Big \|_{CM_ {\frac{q-s}{q}}}
\]
Therefore, by \eqref{EqT1b}, we see that $A_{\mu,s}:A^p_{\alpha}\rightarrow L^q(\Sn)$ is bounded with
\[
\big \|A_{\mu,s} \big \|_{A^p_{\alpha} \rightarrow L^q(\Sn)}=\sup_{\|f\|_{A^p_{\alpha}=1}} \Big \| A_{\mu,s} (f) \Big \|_{L^{q}(\Sn)} \lesssim \sup_{\|f\|_{A^p_{\alpha}=1}}\Big \| \nu_ f \Big \|^{1/s}_{CM_ {\frac{q-s}{q}}},
\]
and the result follows.
\end{proof}

It remains to prove part (c) and (d) in Theorem \ref{th-1}. We start with the case $p>\max(s,q)$.
\begin{theorem}\label{th-03}
Let $\alpha>-1$,  $0< \max\{s,q\}<p<\infty$, and let $\mu$ be a positive Borel measure on $\Bn$.  Then $A_{\mu,s}: A_\alpha^p\rightarrow L^q (\Sn)$ is bounded if and only if
$$
 G_{\mu}(\zeta):=\int_{\Gamma(\zeta)}\rho_{\alpha,\mu}(z,r)^{\frac{p}{p-s}}\frac{dV_\alpha(z)}{(1-|z|)^n}
$$
belongs to $L^{\frac{q(p-s)}{s(p-q)}}(\Sn)$. Moreover, we have
\[
\big \|A_{\mu,s} \big \|_{A^p_{\alpha} \rightarrow L^q(\Sn)} \asymp \big \|G_{\mu} \big \|^{\frac{p-s}{sp}}_{L^{\frac{q(p-s)}{s(p-q)}}(\Sn)}.
\]
\end{theorem}

\begin{proof} We first discuss   the ``if'' part. Since $p> \max\{s,q\}$, we have $\frac{p}{s}>1$ and $\frac{p}{q}>1$. For
 $f\in A^p_\alpha$, \eqref{est-3} and H\"{o}lder's inequality give
 \begin{eqnarray*}
  |A_{\mu,s}(f)(\zeta)|^s\leq C\left(\int_{\widetilde{\Gamma}(\zeta)}|f(z)|^p\frac{d V_\alpha(z)}{(1-|z|)^n}\right)^{\frac{s}{p}}\left(\int_{\widetilde{\Gamma}(\zeta)}\rho_{\alpha}(\mu)(z,r)^{\frac{p}{p-s}}\frac{d V_\alpha(z)}{(1-|z|)^n}\right)^{\frac{p-s}{p}}.
   \end{eqnarray*}
Applying  H\"{o}lder's inequality to the inequality above, we obtain
\begin{eqnarray*}
&& \int_{\Sn} |A_{\mu,s}f(\zeta)|^qd\sigma(\zeta)\leq C\left[\int_{\Sn} \left(\int_{\widetilde{\Gamma}(\zeta)}|f(z)|^p\frac{d V_\alpha(z)}{(1-|z|)^n}\right)d\sigma(\zeta)\right]^{\frac{q}{p}}\\
&&\verb#       #
 \times\left[\int_{\Sn} \left(\int_{\widetilde{\Gamma}(\zeta)}\rho_{\alpha}(\mu)(z,r)^{\frac{p}{p-s}}\frac{d V_\alpha(z)}{(1-|z|)^n}\right)^{\frac{q(p-s)}{s(p-q)}}d\sigma(\zeta)\right]^{\frac{p-q}{p}}.
   \end{eqnarray*}
Thus, \eqref{EqG} yields
\begin{eqnarray*}
 \int_{\Sn} |A_{\mu,s}f(\zeta)|^qd\sigma(\zeta)\leq C\|f\|_{A^p_\alpha}^{q}\cdot
 \|G_\mu\|^{\frac{q(p-s)}{sp}}_{L^{
 \frac{q(p-s)}{s(p-q)}}(\Sn)}.
 \end{eqnarray*}
 That is, $A_{\mu,s}$ is bounded with
 \[
 \big \|A_{\mu,s}\big \|_{A_\alpha^p\to L^q(\Sn)}\lesssim \big \|G_\mu \big \|^{\frac{(p-s)}{sp}}_{L^{
 \frac{q(p-s)}{s(p-q)}}(\Sn)}.
\]

We next prove the ``only if'' part. Let $Z=\{a_k\}$ be an $r$-lattice with $r$  small enough. Consider the test function $F_ t$, defined for $z\in \Bn$ as
 \begin{equation}\label{textf-1}
F_t(z)=\sum_{k}(1-|a_k|)^{\frac{n}{p}}\lambda_k\gamma_k(t)f_k(z),
\end{equation}
where $\lambda=\{\lambda_k\}\in T^p_p(Z)$, with $\gamma_k:[0,1]\rightarrow\{-1, +1\}$ being the Rademacher functions, and
\[
f_k(z)=\frac{(1-|a_k|)^{b}}{(1-\langle z, a_k\rangle)^{b+\frac{n+1+\alpha}{p}}}
\]
with $b>n\max\{1, \frac{1}{p}\}-\frac{n}{p}$. Notice that, $\lambda=\{\lambda_k\}\in T^p_p(Z)$ if and only if $\left\{(1-|a_k|)^{\frac{n}{p}}\lambda_k\right\}\in \ell^p$. For $t\in [0,1]$, Theorem \ref{atomic} shows $F_t\in A^p_\alpha$ with
\[
\|F_t\|_{A^p_\alpha}\leq C\left\|\lambda\right\|_{ T^p_p(Z)}.
\]
It follows from the boundedness of $A_{\mu,s}: A_\alpha^p\to L^q(\Sn)$ that
\[
\begin{split}
\int_{\Sn} \!\left (\int_{\Gamma(\zeta)}\! \Big |\sum_{k}(1-|a_k|)^{\frac{n}{p}}\lambda_k\gamma_k(t)f_k(z),  \Big |^s \frac{d\mu(z)}{(1-|z|^2)^n} \right )^{q/s} \!\!\!d\sigma(\zeta)&=\int_{\Sn} |A_{\mu,s} F_ t(\zeta)|^q \,d\sigma(\zeta)
\\
&\lesssim \big \|A_{\mu,s}\big \|^q_{A_\alpha^p\to L^q(\Sn)}\cdot \left\|\lambda\right\|^q_{ T^p_p(Z)}.
\end{split}
\]
Integrating respect to $t$ in $(0,1)$ we get
\[
\int_{0}^{1} \int_{\Sn} \left\|\sum_{k}(1-|a_k|)^{\frac{n}{p}}\lambda_k\gamma_k(t)f_k\right\|_{\zeta, s}^q \!\!\!d\sigma(\zeta)\,dt\lesssim \big \|A_{\mu,s}\big \|^q_{A_\alpha^p\to L^q(\Sn)}\cdot \left\|\lambda\right\|^q_{ T^p_p(Z)},
\]
where, for $f\in H(\Bn)$, we use the notation
$$
\|f\|_{\zeta, s}=\left(\int_{\Gamma(\zeta)}|f(z)|^s\frac{d \mu(z)}{(1-|z|)^n}\right)^{\frac{1}{s}}.
$$
If $s\geq 1$, then $\|\cdot\|_{\zeta, s}$ is a norm; and $\|\cdot\|^s_{\zeta, s}$ is a quasinorm  on $H(\Bn)$ for $s<1$.
By Kahane's inequality, we have
\begin{eqnarray*}
 \left(\int_{0}^1 \left\|\sum_{k}(1-|a_k|)^{\frac{n}{p}}\lambda_k\gamma_k(t)f_k\right\|_{\zeta, s}^qdt\right)^{1/q} \asymp  \left(\int_{0}^1 \left\|\sum_{k}(1-|a_k|)^{\frac{n}{p}}\lambda_k\gamma_k(t)f_k\right\|_{\zeta, s}^s dt\right)^{1/s},
   \end{eqnarray*}
   and this give
\[
\int_{\Sn} \left(\int_{0}^1 \left\|\sum_{k}(1-|a_k|)^{\frac{n}{p}}\lambda_k\gamma_k(t)f_k\right\|_{\zeta, s}^s dt\right)^{q/s} \!\!\!d\sigma(\zeta)\,dt\lesssim \big \|A_{\mu,s}\big \|^q_{A_\alpha^p\to L^q(\Sn)}\cdot \left\|\lambda\right\|^q_{ T^p_p(Z)}.
\]
Applying  Khinchine's inequality, we obtain
\[
\begin{split}
\int_{\Sn}\!\!\left[\int_{\Gamma(\zeta)}\!\!\left(\sum_{k}(1-|a_k|)^{\frac{2n}{p}}|\lambda_k|^2|f_k(z)|^2\right)^{s/2}\!\!\!\!\frac{d\mu(z)}{(1-|z|)^n}\right]^{q/s}\!\!\!\!\!d\sigma(\zeta)
\lesssim \big \|A_{\mu,s}\big \|^q_{A_\alpha^p\to L^q(\Sn)}\cdot \left\|\lambda\right\|^q_{ T^p_p(Z)}.
\end{split}
\]
  On the other hand, there exist some $\tau>1$ such that $D(z, r)\subset \Gamma (\zeta)$ if $z\in  \Gamma _\tau(\zeta)$. Thus, after an application of Lemma 2.24 in \cite{ZhuBn},
\begin{eqnarray*}
&&\int_{\Gamma(\zeta)}\left(\sum_{k}(1-|a_k|)^{\frac{2n}{p}}|\lambda_k|^2|f_k(z)|^2\right)^{s/2}\frac{d\mu(z)}{(1-|z|)^n}
\\
&&\verb#   #\geq \sum_{a_j\in \Gamma_\tau(\zeta)}\int_{\Gamma(\zeta)\bigcap D_j}\left(\sum_{k}(1-|a_k|)^{\frac{2n}{p}}|\lambda_k|^2|f_k(z)|^2\right)^{s/2}\frac{d\mu(z)}{(1-|z|)^n}
\\
&&\verb#   #\geq \sum_{a_j\in \Gamma_\tau(\zeta)}(1-|a_j|)^{\frac{ns}{p}}|\lambda_j|^s\int_{D_j} |f_j(z)|^s\frac{d\mu(z)}{(1-|z|)^n}
\\
&&\verb#   #\gtrsim \sum_{a_j\in \Gamma_\tau(\zeta)}|\lambda_j|^s\rho_{\alpha,\mu}(a_j,r)\,(1-|a_j|)^{(1+\alpha)\frac{(p-s)}{p}},
   \end{eqnarray*}
where the last inequality is due to the fact that $|f_ j(z)| \asymp (1-|a_ j|)^{-\frac{(n+1+\alpha)}{p}}$ and $(1-|z|)\asymp (1-|a_ j|)$ for $z\in D_ j=D(a_ j,r)$. Therefore, we have
\begin{equation}\label{star}
\int_{\Sn}\!\!\left(\sum_{a_j\in \Gamma(\zeta)}\!\!|\lambda_j|^s\rho_{\alpha,\mu}(a_j,r)(1-|a_j|)^{\frac{(1+\alpha)(p-s)}{p}}  \right)^{q/s}\!\! \!\!\!d\sigma(\zeta)\lesssim \big \|A_{\mu,s}\big \|^q_{A_\alpha^p\to L^q(\Sn)}\cdot \left\|\lambda\right\|^q_{ T^p_p(Z)}.
   \end{equation}

 To prove $G_{\mu}\in L^{\frac{(p-s)q}{s(p-q)}}(\Sn)$, by Lemma \ref{discrete-1}, it is sufficient to show that
\begin{equation}\label{est01}
K_{\mu}(\zeta):= \left(\sum_{a_k\in \Gamma(\zeta)}\rho_{\alpha,\mu}(a_k,2r)^{\frac{p}{p-s}}(1-|a_k|)^{1+\alpha}\right)^{\frac{(p-s)q}{ps}}\in L^{\frac{p}{p-q}}(\Sn).
\end{equation}
Moreover, in this case, we will have
\[
\big \|G_{\mu} \big \|^{\frac{p-s}{sp}}_{L^{\frac{q(p-s)}{s(p-q)}}(\Sn)} \lesssim \big \|K_{\mu} \big \|^{1/q}_{L^{\frac{p}{p-q}}(\Sn)}.
\]
Write $\nu_k=\rho_{\alpha,\mu}(a_k,2r)^{q/s}(1-|a_k|)^{\frac{(1+\alpha)(p-s)q}{ps}}$. Then \eqref{est01} holds if and only if the sequence $\nu=\{\nu_k\}$ belongs to the tent sequence space  $T^{\frac{p}{p-q}}_{\frac{ps}{(p-s)q}}(Z)$. Moreover
$
\big \|K_{\mu} \big \|_{L^{\frac{p}{p-q}}(\Sn)}\asymp \|\nu\|_{T^{\frac{p}{p-q}}_{\frac{ps}{(p-s)q}}(Z)}.
$
  For $t>1$, this is equivalent to  $\nu^{1/t}:=\{\nu_k^{1/t}\}\in T^{\frac{pt}{p-q}}_{\frac{pst}{(p-s)q}}(Z)$ with
\[
  \|\nu^{1/t}\|_{T^{\frac{pt}{p-q}}_{\frac{pst}{(p-s)q}}(Z)}\asymp \|\nu\|^{1/t}_{T^{\frac{p}{p-q}}_{\frac{ps}{(p-s)q}}(Z)} \asymp \big \|K_{\mu} \big \|_{L^{\frac{p}{p-q}}(\Sn)}^{1/t}.
\]
For $t$ large enough, Theorem \ref{TTD1} and Theorem \ref{TTD2} imply
\[
T^{\frac{pt}{p-q}}_{\frac{pst}{(p-s)q}}(Z)=\left(T^{(\frac{pt}{p-q})^{'}}_{(\frac{pst}{(p-s)q})^{'}}(Z)\right)^{*}=\left(T^{t'}_{\frac{st}{st-q}}(Z)\cdot T^{\frac{pt}{q}}_{\frac{pt}{q}}(Z)\right)^{*}.
\]
Take any $\eta=\{\eta_k\}\in T^{(\frac{pt}{p-q})^{'}}_{(\frac{pst}{(p-s)q})^{'}}(Z)$, and factor it as suggested as
\[
\eta_k=\tau_k\lambda_k^{q/t},\quad \tau=\{\tau_k\}\in T^{t'}_{\frac{st}{st-q}}(Z),   \quad  \lambda=\{\lambda_k\}\in T^{p}_{p}(Z),
\]
with
\[
\|\tau \|_{T^{t'}_{\frac{st}{st-q}}(Z)}\cdot \|\lambda \|^{q/t}_{T^p_ p(Z)} \lesssim \|\eta \|_{T^{(\frac{pt}{p-q})^{'}}_{(\frac{pst}{(p-s)q})^{'}}(Z)}.
\]
Without loss of generality, we assume that all sequences are positive. Also, $t>1$ has been taken big enough so that also $st>q$. Using the factorization, the estimate \eqref{EqG} and H\"{o}lder's inequality twice, we get
\begin{eqnarray*}
&&\sum_{a_k\in \Gamma(\zeta)}\eta_k\,\nu_k^{1/t}\,(1-|a_k|)^{n}=\sum_{a_k\in \Gamma(\zeta)}\tau_k\,\lambda_k^{q/t}\,\nu_k^{1/t}\,(1-|a_k|)^{n}
\\
&&\verb#      #\asymp\int_{\Sn}\left(\sum_{a_k\in \Gamma(\zeta)}\tau_k \,\lambda_k^{q/t}\,\nu_k^{1/t}\right)d\sigma(\zeta)
\\
&&\verb#      #\leq \int_{\Sn}\left(\sum_{a_k\in \Gamma(\zeta)}\tau_k^{(st/q)'}\right)^{\frac{st-q}{st}}\left(\sum_{a_k\in \Gamma(\zeta)}\lambda_k^{s}\,\nu_k^{s/q}\right)^{\frac{q}{st}}d\sigma(\zeta)
\\
&&\verb#      #\leq \|\tau\|_{T^{t'}_{\frac{st}{st-q}}(Z)}\left[\int_{\Sn}\left(\sum_{a_k\in \Gamma(\zeta)}\lambda_k^{s}\nu_k^{s/q}\right)^{q/s}d\sigma(\zeta)\right]^{1/t}.
   \end{eqnarray*}
Therefore, by \eqref{star},
\[
\begin{split}
\sum_{a_k\in \Gamma(\zeta)}\eta_k\,\nu_k^{1/t}(1-|a_k|)^{n} &\lesssim \big \|A_{\mu,s} \big \|_{A^p_\alpha\to L^q(\Sn)}^{q/t} \cdot \|\tau\|_{T^{t'}_{\frac{st}{(st-q}}(Z)} \cdot \left\|\lambda\right\|^{q/t}_{ T^p_p(Z)}
\\
& \lesssim \big \|A_{\mu,s} \big \|_{A^p_\alpha\to L^q(\Sn)}^{q/t} \cdot \|\eta \|_{T^{(\frac{pt}{p-q})^{'}}_{(\frac{pst}{(p-s)q})^{'}}(Z)}.
 \end{split}
\]
By the duality of tent sequence spaces, we obtain that $\nu^{1/t}\in  T^{\frac{pt}{p-q}}_{\frac{pst}{(p-s)q}}(Z)$ with
\[
 \|\nu^{1/t}\|_{T^{\frac{pt}{p-q}}_{\frac{pst}{(p-s)q}}(Z)} \lesssim \big \|A_{\mu,s} \big \|_{A^p_\alpha\to L^q(\Sn)}^{q/t}.
\]
Hence, $G_{\mu}\in L^{\frac{(p-s)q}{s(p-q)}}(\Sn)$ with
\[
\big \|G_\mu \big \|^{\frac{p-s}{sp}}_{L^{\frac{(p-s)q}{s(p-q)}}(\Sn)} \lesssim \big \| K_{\mu} \big \|^{1/q}_{L^{\frac{p}{p-q}}(\Sn)} \asymp \Big \|\nu^{1/t}\Big \|^{t/q}_{T^{\frac{pt}{p-q}}_{\frac{pst}{(p-s)q}}(Z)} \lesssim \big \|A_{\mu,s} \big \|_{A^p_\alpha\to L^q(\Sn)}.
\]
finishing the proof of the theorem.
\end{proof}
Finally, the last case in Theorem \ref{th-1}.
\begin{theorem}\label{th-04}
Let $\alpha>-1$, $0<s, p, q<\infty$, and let $\mu$ be a positive Borel measure on $\Bn$. If $q<p\leq s$,   then $A_{\mu,s}: A_\alpha^p\rightarrow L^q (\Sn)$ is bounded if and only if
 $$
 H_{\mu}(\zeta):=\sup\limits_{z\in \Gamma(\zeta)}\rho_{\alpha,\mu}(z,r)(1-|z|)^{\frac{(\alpha+1)(p-s)}{p}}
 $$
belongs to $L^{\frac{pq}{s(p-q)}}(\Sn).$ Moreover, we have
\[
\big \|A_{\mu,s} \big \|_{A^p_{\alpha} \rightarrow L^q(\Sn)} \asymp \big \|H_\mu \big \|^{1/s}_{L^{\frac{pq}{s(p-q)}}(\Sn)}.
\]
\end{theorem}

\begin{proof}
 Suppose $A_{\mu,s}: A_\alpha^p\rightarrow L^q (\Sn)$ is bounded. By Lemma \ref{discrete-2}, we only need to show
\begin{equation}\label{est02}
K_{\mu}:=\int_{\Sn}\sup_{a_k\in \Gamma(\zeta)}\rho_{\alpha,\mu}(a_k,2r)^{\frac{pq}{s(p-q)}}(1-|a_k|)^{\frac{(1+\alpha)(p-s)q}{s(p-q)}}d\sigma(\zeta)<\infty,
\end{equation}
and in this case, we will have $\|H_{\mu}\|^{\frac{pq}{s(p-q)}}_{L^{\frac{pq}{s(p-q)}}(\Sn)}\lesssim K_{\mu}$.
Write $\nu_k=\rho_{\alpha,\mu}(a_k,2r)^{\frac{q}{s}}
(1-|a_k|)^{\frac{(1+\alpha)(p-s)q}{ps}}$. It is obvious that  \eqref{est02} holds if and only if $\nu=\{\nu_k\}\in T^{\frac{p}{p-q}}_{\infty}(Z)$.   For $t>1$, this is equivalent to the statement  $\nu^{1/t}:=\{\nu_k^{1/t}\}\in T^{\frac{pt}{p-q}}_{\infty}(Z)$ with
\[
K_{\mu}\asymp \big \|\nu \big \|^{\frac{p}{p-q}}_{T^{\frac{p}{p-q}}_{\infty}(Z)} \asymp \Big \| \nu^{1/t} \Big \|^{\frac{pt}{p-q}} _{T^{\frac{pt}{p-q}}_{\infty}(Z)}.
\]
For $t$ large enough, by Theorem \ref{TTD1} and Theorem \ref{TTD2}, we have
$$
T^{\frac{pt}{p-q}}_{\infty}(Z)=
\left(T^{(\frac{pt}{p-q})^{'}}_{\varrho}
(Z)\right)^{*}=\left(T^{t'}_{\frac{st}{st-q}}(Z)\cdot T^{\frac{pt}{q}}_{\frac{pt}{q}}(Z)\right)^{*},
$$
since
$$
\frac{st-q}{st}+\frac{q}{pt}=\frac{1}{\varrho}$$ for some $0<\varrho\leq 1$ (because $p\le s$). Take any  $\eta=\{\eta_k\}\in T^{(\frac{pt}{p-q})^{'}}_{\varrho}(Z)$, and factor it as
$
\eta_k=\tau_k\lambda_k^{q/t},$ with  $\tau=\{\tau_k\}\in T^{t'}_{\frac{st}{st-q}}(Z)$ and   $\lambda=\{\lambda_k\}\in T^{p}_{p}(Z).
$
As in the proof of Theorem \ref{th-03}, and because we know that
\eqref{star} also holds for $q<p\le s$, we have
\[
\begin{split}
\sum_{a_k\in \Gamma(\zeta)}|\eta_k| \,|\nu_k^{1/t}|\,(1-|a_k|)^{n}
&\lesssim \|\tau\|_{T^{t'}_{\frac{st}{st-q}}(Z)}\left[\int_{\Sn}\left(\sum_{a_k\in \Gamma(\zeta)}\lambda_k^{s}\nu_k^{s/q}\right)^{q/s}d\sigma(\zeta)\right]^{1/t}
\\
&\lesssim \|\tau\|_{T^{t'}_{\frac{st}{st-q}}(Z)} \cdot  \big \|A_{\mu,s} \big \|_{A^p_\alpha\to L^q(\Sn)}^{q/t}\cdot \|\lambda\|^{q/t}_{ T^p_p(Z)}
\\
&\lesssim \|\eta\|_{T^{(\frac{pt}{p-q})^{'}}_{\varrho}(Z)} \cdot  \big \|A_{\mu,s} \big \|_{A^p_\alpha\to L^q(\Sn)}^{q/t}.
   \end{split}
\]
By duality we get $\nu^{1/t}\in T^{\frac{pt}{p-q}}_{\infty}(Z)$ with $\big \| \nu^{1/t} \big \|_{T^{\frac{pt}{p-q}}_{\infty}(Z)}\lesssim \big \|A_{\mu,s} \big \|_{A^p_\alpha\to L^q(\Sn)}^{q/t}$. Hence
\[
\big \|H_\mu \big \|^{1/s}_{L^{\frac{pq}{s(p-q)}}(\Sn)} \lesssim K_{\mu} ^{\frac{p-q}{pq}} \asymp \Big \| \nu^{1/t} \Big \|^{t/q} _{T^{\frac{pt}{p-q}}_{\infty}(Z)}.\lesssim \big \|A_{\mu,s} \big \|_{A^p_\alpha\to L^q(\Sn)}.
\]

Conversely, for $f\in A^p_\alpha$,
 H\"{o}lder's inequality shows
\begin{eqnarray*}
&& \int_{\Sn} |A_{\mu,s}f(\zeta)|^qd\sigma(\zeta)\lesssim \int_{\Sn} \left(\int_{\widetilde{\Gamma}(\zeta)}|f(z)|^s\rho_{\alpha}(\mu)(z,r)\frac{d V_\alpha(z)}{(1-|z|)^n}\right)^{\frac{q}{s}}d\sigma(\zeta)
\\
&&\verb#       #
\lesssim\int_{\Sn}\left(\sup_{z\in \widetilde{\Gamma}(\zeta)}\rho_{\alpha}(\mu)(z,r)(1-|z|)^{\frac{(1+\alpha)(p-s)}{p}}\right)^{\frac{q}{s}}
\\
&&\verb#         #\times\left(\int_{\widetilde{\Gamma}(\zeta)}|f(z)|^s(1-|z|)^{\frac{(1+\alpha)(s-p)}{p}-n}d V_\alpha(z)\right)^{\frac{q}{s}}d\sigma(\zeta)
\\
&&\verb#       #
\lesssim \left[\int_{\Sn}\left(\sup_{z\in \widetilde{\Gamma}(\zeta)}\rho_{\alpha}(\mu)(z,r)(1-|z|)^{\frac{(1+\alpha)(p-s)}{p}}\right)^{\frac{pq}{s(p-q)}}d\sigma(\zeta)\right]^{\frac{p-q}{p}}
\\
&&\verb#         #\times\left[\int_{\Sn}\left(\int_{\widetilde{\Gamma}(\zeta)}|f(z)|^s(1-|z|)^{\frac{(1+\alpha)(s-p)}{p}-n}d V_\alpha(z)\right)^{\frac{p}{s}}d\sigma(\zeta)\right]^{\frac{q}{p}}.
   \end{eqnarray*}
   On the other hand, we have
\[
\begin{split}
\int_{\widetilde{\Gamma}(\zeta)} |f(z)|^s(1-|z|)^{\frac{(1+\alpha)(s-p)}{p}-n} &d V_\alpha(z)\leq \sum_{a_k\in \widetilde{\Gamma}(\zeta)}\int_{D_k}|f(z)|^s(1-|z|)^{\frac{(1+\alpha)(s-p)}{p}-n}  d V_\alpha(z)
\\
&\leq\sum_{a_k\in \widetilde{\Gamma}(\zeta)}(1-|a_k|)^{\frac{(1+\alpha)s}{p}}\left(\sup_{z\in D_k}|f(z)|\right)^s
\\
&\lesssim \sum_{a_k\in \widetilde{\Gamma}(\zeta)}(1-|a_k|)^{-\frac{ns}{p}} \left(\int_{\widetilde{D}_k}|f(z)|^pd V_\alpha(z)\right)^{\frac{s}{p}}.
\end{split}
\]
Since $\frac{p}{s}\leq 1$,  we obtain
\[
\begin{split}
\left(\int_{\widetilde{\Gamma}(\zeta)}\!\!|f(z)|^s(1-|z|)^{\frac{(1+\alpha)(s-p)}{p}-n} d V_\alpha(z)\right)^{\frac{p}{s}} &
\lesssim\sum_{a_k\in \widetilde{\Gamma}(\zeta)}(1-|a_k|)^{-n} \int_{\widetilde{D}_k}|f(z)|^pd V_\alpha(z)
\\
&\lesssim \int_{\widetilde{\widetilde{\Gamma}}(\zeta)}(1-|z|)^{-n} |f(z)|^pd V_\alpha(z).
   \end{split}
\]

Recall that different apertures define the same tent spaces with equivalent quasinorms. The assumption that $H_\mu\in L^{\frac{pq}{s(p-q)}}(\Sn)$ tells that
\begin{eqnarray*}
 \int_{\Sn} |A_{\mu,s}f(\zeta)|^qd\sigma(\zeta)\lesssim \big \|H_\mu \big \|^{q/s}_{L^{\frac{pq}{s(p-q)}}(\Sn)}\left[\int_{\Sn} \int_{\Gamma(\zeta)}(1-|z|)^{-n} |f(z)|^pd V_\alpha(z)d\sigma(\zeta)\right]^{\frac{q}{p}}.
   \end{eqnarray*}
Since, by \eqref{EqG},
\begin{eqnarray*}
\int_{\Sn} \int_{\Gamma(\zeta)} |f(z)|^p\frac{d V_\alpha(z)}{(1-|z|)^{n}}d\sigma(\zeta)\asymp\int_{\Bn}|f(z)|^pd V_\alpha(z),
   \end{eqnarray*}
we obtain
\begin{eqnarray*}
 \big \| A_{\mu,s} f \big \|^q_{L^q(\Sn)} \lesssim \|f\|_{A^p_\alpha}^{q}\cdot \big \|H_\mu \big \|^{\frac{q}{s}}_{L^{\frac{pq}{s(p-q)}}(\Sn)}.
   \end{eqnarray*}
That is, $A_{\mu,s}$ is bounded from $A_\alpha^p$ to $L^q (\Sn)$ with $\big \|A_{\mu,s} \big \|_{A^p_{\alpha} \rightarrow L^q(\Sn)} \lesssim \big \|H_\mu \big \|^{1/s}_{L^{\frac{pq}{s(p-q)}}(\Sn)}$, which completes the proof.
\end{proof}


\end{document}